\documentclass[12pt]{amsart}
\usepackage{amssymb,amsmath}
\makeatletter
\def\@cite#1#2{{\m@th\upshape\bfseries%
[{#1\if@tempswa{\m@th\upshape\mdseries, #2}\fi}]}}
\makeatother
\theoremstyle{plain}
\newtheorem{thm}{Theorem}[section]
\newtheorem{cor}[thm]{Corollary}

\newtheorem{lem}[thm]{Lemma}

\theoremstyle{definition}

\newtheorem{eg}[thm]{Example}

\newcommand{\Fn}{\bF_n^+}
\newcommand{\Fock}{\ell^2(\Fn)}

\newcommand{\lip}{\langle}
\newcommand{\rip}{\rangle}
\newcommand{\ip}[1]{\left\lip #1 \right\rip}
\newcommand{\bip}[1]{\big\lip #1 \big\rip}

\newcommand{\mt}{\varnothing}

\newcommand{\ol}{\overline}

\newcommand{\upminus}{\raise.9ex\hbox{-\!}}

\newcommand{\sot}{\textsc{sot}}
\newcommand{\wot}{\textsc{wot}}

\DeclareMathOperator*{\sotstarlim}{\textsc{sot}*--lim}

\newcommand{\sotsum}{\textsc{sot--}\!\!\sum}

\newcommand{\bB}{{\mathbb{B}}}
\newcommand{\bC}{{\mathbb{C}}}
\newcommand{\bD}{{\mathbb{D}}}

\newcommand{\bF}{{\mathbb{F}}}

 \newcommand{\B}{{\mathcal{B}}}
 \newcommand{\C}{{\mathcal{C}}}

\renewcommand{\H}{{\mathcal{H}}}

\renewcommand{\P}{{\mathcal{P}}}

 \newcommand{\U}{{\mathcal{U}}}

\newcommand{\ep}{\varepsilon}

\renewcommand{\phi}{\varphi}

\newcommand{\upchi}{{\raise.35ex\hbox{$\chi$}}}

\newcommand{\fA}{{\mathfrak{A}}}

\newcommand{\fJ}{{\mathfrak{J}}}
\newcommand{\fK}{{\mathfrak{K}}}
\newcommand{\fL}{{\mathfrak{L}}}
\newcommand{\fM}{{\mathfrak{M}}}

\newcommand{\fS}{{\mathfrak{S}}}

\newcommand{\fW}{{\mathfrak{W}}}

\newcommand{\AD}{\mathrm{A}(\mathbb{D})}

\newcommand{\linf}{\ell^\infty }

\newcommand{\lone}{\ell^1}
\newcommand{\ltwo}{\ell^2}

\newcommand{\AND}{\text{ and }}

\newcommand{\qand}{\quad\text{and}\quad}

\newcommand{\qforal}{\quad\text{for all}\quad}

\newcommand{\dist}{\operatorname{dist}}

\newcommand{\id}{\operatorname{id}}

\newcommand{\re}{\operatorname{Re}}

\newcommand{\spn}{\operatorname{span}}

\begin{document}

\title{Operator Algebras with Unique Preduals}
%
\author[K.R.Davidson]{Kenneth R. Davidson}
\thanks{First author partially supported by an NSERC grant}
\address{Pure Math.\ Dept.\\U. Waterloo\\Waterloo, ON\;
N2L--3G1\\CANADA}
\email{krdavids@uwaterloo.ca}
\author[A.Wright]{Alex~Wright}
\address{Pure Math.\ Dept.\\U. Waterloo\\Waterloo, ON\;
N2L--3G1\\CANADA}
\email{alexmwright@gmail.com}
\date{}

\begin{abstract}
We show that every free semigroup algebras has a (strongly) unique Banach space predual. 
We also provide a new simpler proof that a weak-$*$ closed unital operator operator algebra
containing a weak-$*$ dense subalgebra of compact operators has a unique
Banach space predual. 
\end{abstract}
\maketitle

\section{Introduction}\label{S:intro}

A famous theorem of Sakai \cite{Sakai} showed that C*-algebras which are dual spaces are
von Neumann algebras, and the techniques showed in addition that the predual of a von Neumann algebra is unique (up to isometric isomorphism).
This generalized a result of Grothendieck \cite{Gr} that $L^\infty(\mu)$ has a unique predual.
Ando \cite{Ando} showed that the algebra $H^\infty$ of bounded analytic functions on
the unit disk also has a unique predual.  More recently, 
Ruan \cite{Ruan} showed that an operator algebra with a weak-$*$  dense subalgebra
of compact operators has a unique operator space predual.
He points out that some general Banach space methods of Godefroy \cite{Go,GL} in fact
imply that such algebras have a unique Banach space predual.
Also, Effros, Ozawa and Ruan \cite{EOR} have shown that W*TROs
(corners of von Neumann algebras) have unique preduals as well.

In this note, we show that every free semigroup operator algebra has a unique predual.
A free semigroup algebra is the \wot-closed unital algebra generated by $n$ isometries
with pairwise orthogonal ranges.  The prototypes are the non-commutative analytic
Toeplitz algebras, $\fL_n$, given by the left regular representation of the free semigroup $\Fn$
of words in an alphabet of $n$ letters \cite{Pop_fact,Pop_beur,DP1}.
The case $n=1$ is just $H^\infty$, which follows from Ando's Theorem.
Our proof deals with $n \ge 2$.
Once the result is established for $\fL_n$, the general case follows from
the Structure Theorem for free semigroup algebras \cite{DKP} and the result
mentioned above of Effros, Ozawa and Ruan.

It is an open problem whether $H^\infty(\Omega)$ has a unique predual when $\Omega$
is a domain in $\bC^n$, even for the bidisk $\bD^2$ or the unit ball $\bB_2$ of $\bC^2$.
In a number of ways, the algebras $\fL_n$ have proven to be more tractable that their
commutative counterparts when it comes to finding analogues of classical results
for $H^\infty$ in dimension one.  This predual result is another case in point.

A weak-$*$ closed operator algebra $\fA$ is called \textit{local} if the ideal of compact
operators in $\fA$ is weak-$*$ dense in $\fA$.
We also provide a new simpler proof that a local operator algebra has a unique predual
which is inspired by Ando's proof for $H^\infty$.
For $\B(\H)$, the manipulations involving approximate identities can be omitted.
So this provides an alternative to invoking Sakai's Theorem in this case.
However we also provide another very simple proof for $\B(\H)$ which relies
neither on positivity (like Sakai) nor on the density of the compacts.

In Banach space theory, there is an extensive literature on the topic of unique preduals.
We refer the reader to a nice survey paper of Godefroy \cite{Go}.
For example, if $X$ is a dual space which does not contain an isomorphic copy of $\ell^1$,
then the predual is unique.
Also smoothness conditions on the predual $X_*$ such as a locally uniformly convex norm
or the Radon--Nikodym property imply that it is the unique predual of $X$.
These properties do not often apply to algebras of operators.

Another observation due to Godefroy and Talagrand \cite{GT} is that Banach spaces
with property (X) have unique preduals.  
This technical condition will be defined in the next section. 
In the proofs of Sakai and Ando mentioned above, 
it is a property very close to this which is exploited to establish uniqueness.  
It implies, for example, that if $X$ is an M-ideal in $X^{**}$,
then $X^*$ is the unique predual of $X^{**}$ \cite[p.148]{HWW}.
This is the case for operator algebras with a weak-$*$ dense ideal of compact operators.
Recently, Pfitzner \cite{Pf} has has generalized this by showing that if $X_*$ is an L-summand 
in $X^*$, then $X_*$ has property (X) and so is the unique predual of $X$.
Another basic class with unique predual are the spaces of operators $\B(X,Y)$
where $X$ and $Y$ are Banach spaces with the Radon--Nikodym property
due to Godefroy and Saphar \cite{GS}.
This includes spaces which are separable dual spaces, and 
all reflexive Banach spaces.

Nevertheless, in operator algebras, the literature on unique preduals is rather limited,
and the main results have all been mentioned above.

A Banach space $X$ has a strongly unique predual if there is a unique subspace
$E$ of $X^*$ for which $X=E^*$.
All known examples of Banach spaces with unique predual actually have
a strongly unique predual \cite{Go}.
This is the case in our examples as well.

We thank Gilles Godefroy for some helpful comments on this paper.

\section{Background}\label{S:back}

If $X$ is a dual Banach space, then any predual $E$ sits in a canonical manner
as a subspace of the dual $X^*$.
Let $\sigma(X,E)$ denote the weak-$*$ topology on $X$ induced from $E$.
$E$ has two evident properties which are characteristic:
\begin{enumerate}
\item $E$ norms $X$: $\sup\{ |\phi(x)| : \phi \in E,\, \|\phi\| \le 1 \} = \|x\|$.
\item The closed unit ball of $X$ is compact in the $\sigma(X,E)$ topology.
\end{enumerate}
The latter property is a consequence of the Banach-Aloaglu Theorem.

Conversely, if $E$ is a subspace of $X^*$ with these properties,
then $X$ sits isometrically as a subspace of $E^*$ by (1).
By (2), the closed ball $\ol{b_1(X)} := \ol{b_1(0)}$ of $X$ is weak-$*$ compact in $E^*$.
Therefore by the Krein--Smulyan Theorem, $X$ is weak-$*$ closed in $E^*$.
However, as $E$ is a subspace of $X^*$, the annihilator of $X$ in $E$ in $\{0\}$.
Hence $X=E^*$.
Thus we see that these two properties characterize the preduals of $X$.

In any weak-$*$ topology on $X$, closed balls $\ol{b_r(x)}$ are compact
for $x\in X$ and $r\ge0$.
Also, addition is always weak-$*$ continuous.  So finite sums of
closed balls are also \textit{universally weak-$*$ compact}.
This can sometimes be used to show that certain functionals are
\textit{universally weak-$*$ continuous}, meaning that they belong
to every predual of $X$.

\begin{eg}
It is very easy to see that $\linf$ has a unique predual, namely $\lone$.
Let $e_n$ denote the sequence with a $1$ in the $n$th coordinate, and $0$ elsewhere.
And let $\delta_n$ be the element of $\lone$ which evaluates the $n$th coordinate.
Observe that 
\[  \ol{b_1(e_n)} \cap \ol{b_1(-e_n)} = \ol{b_1(\ker\delta_n))}. \]
Hence $\ol{b_1(\ker\delta_n))}$ is universally weak-$*$ compact.
By the Krein--Smulyan Theorem, $\ker\delta_n$ is universally weak-$*$ closed.
So $\delta_n$ lies in every predual of $\linf$.  But these functionals span $\lone$,
and hence it is the unique predual.
\end{eg}

\begin{eg}
A similar, but somewhat more involved, argument shows that $\B(\H)$ has a unique predual.
Let $\H$ be an infinite dimensional Hilbert space, and consider unit vectors $x,y\in\H$.
Write $xy^*$ for the rank one operator $xy^*(z) = \ip{z,y}x$.
Observe that 
\begin{align*}
 \C_x^y &:=\ol{b_1(xy^*)} \cap \ol{b_1(-xy^*)} \\&= \{T \in \B(\H) : Ty=T^*x=0 \AND \|T\|\le 1 \}.
\end{align*}
Indeed, $Ty \in \ol{b_1(x)} \cap \ol{b_1(-x)} = \{0\}$ and similarly, $T^*x=0$;
so the result follows.
Pick a unit vector $z$ orthogonal to both $x,y$.
Then a simple calculation shows that
\[
 (\C_x^y + \C_z^y + \C_x^z) \cap \ol{b_1(\B(\H))} = 
 \{T \in \B(\H) : \ip{Ty,x}=0 \AND \|T\|\le 1 \}.
\]
Arguing as before, the functional $(yx^*)(T) = \ip{Ty,x}$ is universally weak-$*$ continuous.
But these functionals span the trace class operators $\fS_1$, the standard predual of $\B(\H)$.
Therefore $\fS_1$ is the unique predual of $\B(\H)$.
\end{eg}

\begin{eg}
There are \wot-closed operator algebras which do not have unique preduals.
The basic point is that being an operator algebra is not restrictive.
If we put any \wot-closed subspace of $\B(\H)$ in the $1,2$ entry of $2\times2$
matrices over $\B(\H)$, then we have an operator algebra.  Adding in the scalar
operators (on the diagonal) will not essentially change the Banach space
characteristics, but will yield a unital algebra.
In particular, let's put $\linf$ into $\B(\H)$ as the diagonal operators, and place it
in the $1,2$ entry.  Every dual space $X^*$ with separable predual can be
isometrically imbedded into $\linf$ as a weak-$*$ closed subspace.
The weak-$*$ and \wot-topologies coincide on $\linf$.  So this procedure
yields a \wot-closed algebra.  If we do this for $X^* = \lone$,
we obtain the desired example.
\end{eg}

A series $(x_n)$ in a Banach space $X$ is universally weakly Cauchy if
$\sum_{n\ge1}| \phi(x_n) | < \infty$ for every $\phi\in X^*$.
If $X=E^*$, define $C(E)$ to be the set of all functionals $\phi\in X^*=E^{**}$
with the property that for every universally weakly Cauchy series $(x_n)$ in $X$,
\[ \phi \big( w^*\!\!-\!\lim \sum_{i=1}^n x_i \big) = \sum_{i=1}^\infty \phi(x_n) .\]
Evidently this contains $E$.  But Godefroy and Talagrand \cite{GT} show that
$C(E)$ contains every predual of $X$, and this space does not depend on the
choice of the predual $E$.
They say that $X$ has property (X) if $C(E)=E$.  Evidently this immediately
implies that $X$ has a unique predual.

A similar property was established by Sakai for a von Neumann algebra $\fM$.
He shows that a state $\phi$ on $\fM$ belongs to $\fM_*$ if and only if it satisfies: 
whenever $(P_n)$ are pairwise orthogonal projections in $\fM$ such that 
$\sotsum P_n = I$, then $\sum \phi(P_n) = 1$.
In Section~\ref{S:Compacts}, we use a similar property to establish unique preduals
for algebras with sufficiently many compact operators.

\pagebreak[4]
\section{Free Semi-group Algebras}\label{S:fsg}

A free semigroup algebra is a \wot-closed unital operator algebra
generated by $n$ isometries $S_1,\dots,S_n$ with pairwise orthogonal range.
We allow $n=\infty$.
The prototype is obtained from the left regular representation of the free
semigroup $\Fn$ of all words in an alphabet of $n$ letters.
The operators $L_v$, for $v\in\Fn$, act on the Fock space $\ltwo(\Fn)$,
with orthonormal basis $\{\xi_w : w \in \Fn \}$, by  $L_v \xi_w = \xi_{vw}$.
The algebra $\fL_n$ generated by $L_1,\dots,L_n$ is called the
\textit{noncommutative analytic Toeplitz algebra} because the case $n=1$
yields the analytic Toeplitz algebra isomoetrically isomorphic to $H^\infty$,
and because these algebras share many similar properties 
(see \cite{Pop_fact,Pop_beur,DP1,DP2}).
The standard predual of $\fL_n$ is the space $\fL_{n*}$ of all weak-$*$ continuous
linear functionals on $\fL_n$.
See \cite{D_surv} for an overview of these algebras.

Let $|w|$ denote the length of the word $w$.
Note that the operators $\{ L_w : |w| = k \}$ are isometries with pairwise orthogonal ranges.
Thus $\spn \{ L_w : |w| = k \}$ is isometric to a Hilbert space.

An element $A \in \fL_n$ is determined by $A\xi_\mt = \sum a_w L_w$.
We call the series $\sum a_w L_w$ the Fourier series of $A$.
As in classical harmonic analysis, this series need not converge.
However the Cesaro means do converge in the strong operator topology to $A$ \cite{DP1}.
The functional $\phi_w(A) = \ip{A\xi_\mt,\xi_w}$ reads off the $w$-th Fourier coefficient,
and it is evidently weak-$*$ continuous.
The ideal $\fL_n^0 = \ker\phi_\mt$ consists of all elements with 0 constant term, and
is the \wot-closed ideal generated by $L_1,\dots,L_n$.
The powers $(\fL_n^0)^k$ are the ideals of elements for which $a_w=0$ for all $|w| < k$.

We first deal with the noncommutative analytic Toeplitz algebras, $\fL_n$, for $n\ge2$.
This first lemma is motivated by the fact that for a vector $v$ in a Hilbert space $\H$, 
$\bigcap_{\lambda\in\bC}\ol{b_{\sqrt{1+|\lambda|^2}}(\lambda v)}=\bC v^\perp\cap\ol{b_1}$.

\begin{lem}\label{P:ideal_C_w}
For $n\ge2$ and $k \ge 0$, the ideal $(\fL_n^0)^k$
is universally weak-$*$ closed in $\fL_n$.
\end{lem}

\begin{proof}
Fix a word $w\in\Fn$.
Define $C_w=\bigcap_{\lambda\in\bC}\ol{b_{\sqrt{1+|\lambda|^2}}(\lambda L_w)}.$
By the remarks in the previous section, this is universally weak-$*$ compact.
We will first establish that
\[ C_w = \{ A\in\fL_n : \|A\| \le 1 \AND L_w^* A = 0 \} .\]
That is, a contraction $A$ belongs to $\C_w$ if and only if
$A$ and $L_w$ have orthogonal ranges. 
Observe that a contraction $A$ lies in $C_w$ if and only if for all $\lambda\in\bC$ and
all $\xi\in\Fock$ with $\|\xi\|=1$,
\begin{align*}
   1 + |\lambda|^2 &\ge \| \lambda L_w \xi  - A\xi \|^2 \\
   &= \| \lambda L_w\xi \|^2 - 2\re \lip A\xi, \lambda L_w \xi \rip +\|A\xi\|^2\\
   &= |\lambda|^2 - 2\re\lip \ol{\lambda} L_w^*A \xi, \xi\rip + \|A\xi\|^2 .
\end{align*}
If $L_w^* A = 0$ this inequality is clearly satisfied, so $A$ belongs to $C_w$.

Conversely, if $A\in C_w$, by picking the sign of $\lambda$ appropriately, we obtain that
\[
 1 + |\lambda|^2 \ge |\lambda|^2 + 2|\lambda| |\lip L_w^* A\xi,\xi\rip| + \|A\xi\|^2 .
\]
Letting $|\lambda|$ tend to $\infty$, we see that $\lip L_w^* A\xi,\xi\rip = 0$ for all $\xi$.
By the polarization identity, $L_w^* A = 0$.

It follows that $D_{k,i} := \bigcap_{|w|=k-1} C_{wi}$ is universally weak-$*$
compact for any $i$.
If $A \in \fL_n$ has a Fourier series $A \sim \sum  a_vL_v$ and lies in $D_{k,i}$,
then we claim that $a_v = 0$ if $|v| < k$ or if $v$ has the form $wiv'$ for $|w|=k-1$.
Indeed, if $|v|<k$, choose any word $v'$, possibly empty, so that $|vv'i|=k$.
Then since the range of $A$ is orthogonal to the range of $L_{vv'i}$,
\[
 0 = \ip{A \xi_{v'i}, L_{vv'i}\xi_\mt} = \bip{\sum a_v \xi_{vv'i}, \xi_{vv'i}} = a_v .
\]
Similarly if $v=wiv'$ for $|w|=k-1$, then
since the range of $A$ is orthogonal to the range of $L_{wi}$,
\[
 0 = \ip{A \xi_\mt, L_{wi}\xi_{v'}} = \bip{\sum a_v \xi_v, \xi_v} = a_v .
\]
Conversely, it is evident that any such $A$ has range orthogonal to all
ranges $L_{wi}$ for $|w|=k-1$.  So if it is a contraction, it will lie in $D_{k,i}$.

Since addition is always weak-$*$ continuous, we obtain that
\[
 D_k = \ol{b_1(\fL_n)} \cap (D_{k,1}+D_{k,2})
\]
is also universally weak-$*$ compact.  We claim that $D_k = \ol{b_1(( \fL_n^0)^k)}$.
For $A$ to lie in either $D_{k,i}$, the Fourier coefficients $a_v=0$ for $|v|<k$.
So this persists in the sum, and hence $D_k$ is contained in $\ol{b_1( (\fL_n^0)^k)}$.

Conversely, if $A\in\ol{b_1( (\fL_n^0)^k)}$, by \cite[Lemma~2.6]{DKP}
there is a factorization $A=\sum_{|w|=k} L_wA_w$ where $A_w \in \fL_n$.
Moreover, this factors as $LC$ where $L$ is the row operator with coefficients
$L_w$ for $|w|=k$ and $C$ is the column operator with coefficients $A_w$.
Since $L$ is an isometry, we have $\|C\| = \|A\| \le 1$.
Define
\[
 B_1 = \sum_{i\ge 2} \sum_{|w|=k-1} L_{wi}A_{wi} \qand B_2 = \sum_{|w|=k-1} L_{w1}A_{w1} .
\]
Then it follows that both $B_i$ are contractions in $(\fL_n^0)^k$.
Moreover, $B_1 \in D_{k,1}$ and $B_2 \in D_{k,2}$.
Therefore $A = B_1+B_2$ belongs to $D_k$ as claimed.

We have shown that $\ol{b_1(( \fL_n^0)^k)}$ is universally weak-$*$ compact.
Thus by the Krein--Smulyan Theorem, $(\fL_n^0)^k$ is universally weak-$*$ closed.
\end{proof}

\begin{cor} \label{C:Fourier}
The functionals $\phi_w$, for $w \in \Fn$,  are universally weak-$*$ continuous for all $n\ge2$.
\end{cor}

\begin{proof}
If $\fL_n = E^*$, let $E_k = ((\fL_n^0)^k)_\perp$ be the annihilator of $(\fL_n^0)^k$ in $E$.
Since $(\fL_n^0)^k$ is $\sigma(\fL_n,E)$ closed, $(\fL_n^0)^k = E_k^\perp$
and $(\fL_n^0)^k \simeq (E/E_k)^*$.
Therefore $(\fL_n^0)^k/(\fL_n^0)^{k+1} \simeq (E_k/E_{k+1})^*$.

Now $(\fL_n^0)^k/(\fL_n^0)^{k+1}$ is isometrically isomorphic to
the subspace \linebreak\mbox{$\spn\{ L_w : |w|=k\}$.}
Indeed, the elements of this quotient have the form $\sum_{|w|=k} a_wL_w + (\fL_n^0)^{k+1}$.
So the norm is bounded above by
\[ \|\sum_{|w|=k} a_wL_w\| = \|(a_w)_{|w|=k}\|_2 .\]
On the other hand, it is bounded below by
\[
 \inf_{A\in (\fL_n^0)^{k+1}} \big\| \big( \sum_{|w|=k} a_wL_w + A \big)\xi_\mt \big\|
 = \big\| \sum_{|w|=k} a_w \xi_w  \big\|  =  \|(a_w)_{|w|=k}\|_2
\]
Hence this quotient is a Hilbert space.

Since a Hilbert space is reflexive, its dual is $E_k/E_{k+1}$.
Therefore $\fL_n/(\fL_n^0)^{k+1}$ is reflexive with dual $E/E_{k+1}$.
Since the functionals $\phi_w$ for $|w|\le k$ are continuous on this quotient,
they are all $\sigma(\fL_n,E)$ continuous.
\end{proof}

\begin{thm}\label{T:Ln_unique_predual}
   $\fL_n$ has a unique predual for $n\ge2$.
\end{thm}

\begin{proof}
Recall that the standard predual $\fL_{n*}$ of $\fL_n$ consists of the weak-$*$ continuous linear
functionals of the form $[xy^*]$.
Clearly taking $x$ and $y$ to be in the algebraic span of $\{\xi_w : w \in \Fn\}$
is norm dense in the predual.
However, $[\xi_v\xi_w^*] = [\xi_\mt(L_v^*\xi_w)^*]$.
So the span of the functionals $[\xi_\mt\xi_w^*]$ is norm dense in the predual.
Since $ [\xi_\mt\xi_w^*] = \phi_w$ is universally weak-$*$ continuous
by Corollary~\ref{C:Fourier}, it follows that the standard predual is universally
weak-$*$ continuous.

No two preduals are comparable; so it follows that $\fL_{n*}$ is the strongly unique
predual of $\fL_n$.
\end{proof}

To deal with the case of a general free semigroup algebra, we require a simple lemma.
The dual space of an operator algebra $\fA$ is a bimodule over $\fA$ with the
natural action $(A\phi B)(T) = \phi(BTA)$.

\begin{lem}\label{L:proj_functional}
Let $P$ be an orthogonal projection in an operator algebra $\fA$.
Then for $\phi\in\fA^*$,
\[  \|\phi\|^2 \ge \| P\phi\|^2 + \| P^\perp \phi \|^2 .\]
\end{lem}

\begin{proof}
Find $A=AP$ and $B=BP^\perp$ in $\fA$ of norm 1 so that $\phi(A)$ and $\phi(B)$ are real,
and we have the approximations
\[ \phi(A)= (P\phi)(A) \approx \|P\phi\|   \qand \phi(B) = P^\perp \phi(B) \approx \|P^\perp \|  . \]
Consider $T=\cos \theta A+\sin\theta B$.
Note that
\[ \|T\|^2=\|TT^*\| = \|\cos^2AA^*+\sin^2 BB^*\| \le 1 .\]
Compute
\[
 \phi(T) = \cos\theta\phi(A)+\sin\theta\phi(B)
 = ( \cos\theta,\sin\theta)\cdot( \phi(A), \phi(B)).
\]
Choosing $\theta$ so $( \cos\theta,\sin\theta)$ is parallel to $( \phi(A), \phi(B))$,
we obtain
\[ \|\phi\|^2 \ge  |\phi(T)|^2 = \phi(A)^2+\phi(B)^2 \approx \| P\phi\|^2 + \| P^\perp \phi \|^2 .\qedhere \]
\end{proof}

Now the general free semigroup algebra case follows from the structure theory of these algebras.

\begin{thm} \label{T:FSG_unique_predual}
Every free semigroup algebra $\fS$ has a strongly unique predual.
\end{thm}

\begin{proof}
We invoke the Structure Theorem for free semigroup algebras \cite{DKP}.
If $\fS$ is a von Neumann algebra, then the result follows from Sakai's Theorem \cite{Sakai}.
Otherwise, the \wot-closed ideal $\fS_0$ generated by $\{S_1,\dots,S_n\}$ is proper.
Let $\fJ = \bigcap_{k\ge1} \fS_0^k$.
This is a \wot-closed ideal of $\fS$, and is a left ideal in
the von Neumann algebra $\fW$ generated by $\{S_1,\dots,S_n\}$.
There is a projection $P\in\fS$ so that $\fJ = \fW P$, $P^\perp\H$ is invariant
for $\fS$, and $\fS|_{P^\perp\H}$ is completely isometrically isomorphic and weak-$*$
homeomorphic to $\fL_n$.

Now, define
$C_{P^\perp}=\bigcap_{\lambda\in\bC}\ol{B_{\sqrt{1+|\lambda|^2}}(\lambda P^\perp)}.$
We claim $C_{P^\perp}= \ol{b_1(\fJ)}$.
By the calculation of $C_\emptyset$ in $\fL_n$ in Lemma~\ref{P:ideal_C_w},
$P^\perp C_{P^\perp} P^\perp = C_\mt = \{0\}$.
So $C_{P^\perp} \subset \fJ \cap \ol B_1(\fL_n) = b_1(\fJ)$.
Conversely,
\begin{align*}
 \left\| \begin{bmatrix} A & 0 \\  B & \lambda I \end{bmatrix} \right\|^2 &=
 \left\| \begin{bmatrix} A & 0 \\  B & \lambda I \end{bmatrix}
 \begin{bmatrix} A^* & B^* \\  0 & \ol\lambda I \end{bmatrix} \right\| \\ &\le
 \left\| \begin{bmatrix} A & 0 \\ B & 0 \end{bmatrix}
 \begin{bmatrix} A^* & B^* \\  0 & 0 \end{bmatrix} \right\| +
 \left\| \begin{bmatrix} 0 & 0 \\  0 & |\lambda|^2 I \end{bmatrix} \right\| \\&=
 \left\| \begin{bmatrix} A & 0 \\ B & 0 \end{bmatrix} \right\|^2
 + |\lambda|^2 .
\end{align*}
Thus, we see that $C_{P^\perp}$ contains $\ol{b_1(\fJ)}$.
Whence $C_{P^\perp} = \ol{b_1(\fJ)}$.

Now $C_{P^\perp}$ is universally weak-$*$ compact.
By the Krein-Smulyan Theorem, $\fJ = \spn C_{P^\perp}$ is universally weak-closed.
Note, $\fJ$ is a W*TRO, and therefore has a strongly unique predual \cite{EOR}.

Let $E$ be a predual of $\fS$.
Then the predual of $\fS/\fJ$ is
\[ E_0 = \{ \phi\in E : \phi|_\fJ = 0 \} .\]
Since  $\fS/\fJ$ is isomorphic to $\fL_n$,  Theorem~\ref{T:Ln_unique_predual} 
implies that $E_0$ coincides with the weak-$*$ continuous functionals on $\fL_n$.
Because the isomorphism of $\fS|{P^\perp\H}$ to $\fL_n$
is a weak-$*$ homeomorphism, $E_0$ coincide with the weak-$*$ continuous functionals
on $\fS|_{P^\perp \H}$.

The predual of $\fJ$ is $E/E_0$.
If $\phi\in E$, then $\|\phi+E_0\| = \|\phi|_\fJ\| = \|P\phi\|$.
Clearly every functional $\psi\in E_0$ has $\psi = P^\perp\psi$.
Therefore by Lemma \ref{L:proj_functional}, we see that
\[ \|P\phi\|^2 = \|\phi+E_0\|^2 \ge \|P\phi\|^2 + \dist(P^\perp\phi, E_0)^2 .\]
Hence $P^\perp\phi\in E_0$ and so $P\phi\in E$.
It follows that $E_0 = P^\perp E$ and $E/E_0 \simeq PE$.
Hence $\fS P^\perp = (PE)^\perp$ is also $\sigma(\fS_n, E)$ closed.

As $PE$ is the unique predual of $\fS P$ and $P^\perp E$ is the unique predual
of $\fS P^\perp$, both consisting of the weak-$*$ continuous functionals,
we deduce that $E$ necessarily coincides with the weak-$*$ continuous functionals on $\fS$.
So there is a strongly unique predual.
\end{proof}

\section{Operator algebras with many compact operators}\label{S:Compacts}

Suppose that $\fA$ is a local weak-$*$ closed unital sub-algebra of $\B(\H)$,
meaning that $\fA\cap\fK$ is weak-$*$ dense in $\fA$.
We will provide a new proof that $\fA$ has a unique predual using an argument modelled on
Ando's argument \cite{Ando} that $H^\infty$ has a unique predual.

Observe that the weak-$*$ density means that $(\fA \cap \fK)^\perp = \fA_\perp$
in the space $\fS_1$ of  trace class operators.
Thus $(\fA \cap \fK)^* \simeq \fS_1/\fA_\perp \simeq \fA_*$;
and hence $(\fA \cap \fK)^{**} = \fA$.
Thus there is a canonical contractive projection $\P$ of the triple dual, $\fA^*$, onto the
dual space $\fA_*$ given by restriction to $\fA \cap \fK$, 
and considered as a subspace of $\fA^*$.
For $\phi\in\fA^*$, we will write $\P \phi =: \phi_a$ and $(\id - \P)\phi =: \phi_s$.

In fact by \cite{DP}, $\fA \cap \fK$ is an M-ideal in $\fA$.
Therefore $\P$ is an L-projection, meaning that $\|\phi\| = \|\phi_a\|+\|\phi_s\|$.
We will not require this fact in our proof.

By Goldstine's Theorem, the unit ball of $\fA \cap \fK$ is weak* dense in the unit ball of $\fA$.
Since the closed convex sets in \wot\  and \sot$*$ topologies coincide,
the ball of compact operators is also \sot$*$ dense in the ball of $\fA$.
In particular, there is a net (a sequence when $\H$ is separable)
$K_n \in \fA \cap \fK$ with $\|K_n\| \le 1$ and $\sotstarlim K_n = I$.
Evidently, this is a contractive approximate identity for $\fA \cap \fK$.

We require a somewhat better approximate identity.
It is shown in \cite{DP} that the existence of a bounded one-sided approximate identity
implies the existence of a contractive two-sided approximate identity $C_n$ with
the additional property that $\limsup \| I-C_n \| \le 1$.
What we require here is similar, and the argument follows from tricks using
the Riesz functional calculus.

For various interesting examples such as $\B(\H)$ and atomic CSL algebras
(see the end of this section for definitions), the compact operators in the algebra have
a bounded approximate identity consisting of finite rank projections $\{P_n\}$.
In this case, one may take $S_n=C_n=P_n$ in Lemma~\ref{L:GoodApproxIDs},
 and avoid all of the tricky calculations.

\begin{lem}\label{L:GoodApproxIDs}
Let $\fA$ be an operator algebra with a contractive approximate identity $\{K_n\}$.
Then $\fA$ has a contractive approximate identity $\{S_n\}$
and a bounded approximate identity $\{ I-C_n \}$ so that
\[
 \lim_{n\to\infty} \big\| \begin{bmatrix} S_n &  C_n\end{bmatrix} \big\| = 1
 = \lim_{n\to\infty} \bigg\| \begin{bmatrix} S_n \\ C_n\end{bmatrix} \bigg\| .
\]
\end{lem}

\begin{proof}
It is easy to see that for any fixed $i$, $\{K_n^i\}$ is a contractive approximate identity.
So if $p$ is a polynomial $p$ with $p(0)=0$ and $p(1) = 1$, then
$\{p(K_n)\}$ is an approximate identity.  By von Neumann's inequality,
it is bounded by $\|p\|_\infty = \sup_{|z|\le1}|p(z)|$.
Now if $f \in \AD$ satisfies $f(0)=0$ and $f(1)=1$, then it can
be uniformly approximated by such polynomials.  So again von Neumann's inequality
shows that $\{f(K_n\})$ is an approximate identity bounded by $\|f\|_\infty$.

Now $\sin(\pi z/2)$ is analytic, and takes $(-1,1)$ into itself.  It is easy to check that
there is a convex open set $U$ containing $(-1,1)$ on which $| \sin(\pi z/2) | < 1$.
Let $U_\ep = U \cap \{ x+iy :  |y| < \ep \}$; and
let $f_\ep$ be the conformal map of $\bD$ onto $U_\ep$
such that $f_\ep(0) = 0$ and $f_\ep(1)=1$.
We define
\[ g_\ep(z) = \sin \big( \frac\pi 2 f_\ep(z) \big) \qand  h_\ep(z) = \cos\big( \frac\pi 2 f_\ep(z)\big) .\]

Define $S_n = g_\ep(K_n)$  and $C_n = h_\ep(K_n)$.
Then $\{S_n\}$ is a contractive approximate identity,
and $\{I-C_n\}$ is a bounded approximate identity for $\fA$.
By the matrix version of von Neumann's inequality,
\[
 \big\| \! \begin{bmatrix} S_n & C_n\end{bmatrix} \! \big\| \le
 \big\| \! \begin{bmatrix} g_\ep & h_\ep\end{bmatrix} \! \big\| =
 \big\| \sqrt{ |g_\ep|^2 + |h_\ep|^2 } \big\| \le
 \Big( 1 + 2 \sinh^2 \big( \frac{\pi \ep}2 \big) \Big)^{1/2} \!.
\]
Similarly,
\[
 \bigg\| \begin{bmatrix} S_n \\ C_n\end{bmatrix} \bigg\| \le
 \Big( 1 + 2 \sinh^2 \big( \frac{\pi \ep}2 \big) \Big)^{1/2} .
\]

Now use a diagonal argument to let $\ep$ go to $0$ slowly relative to $n$ so as
to still be an approximate identity and obtain the desired norm limit.
\end{proof}

We can now prove the uniqueness of preduals.

\begin{thm}\label{T:AlgebrasWithCompacts}
Let $\fA$ be a weak-$*$ closed unital subalgebra of $\B(\H)$ with a weak-$*$ dense
subalgebra of compact operators.  Then $\fA$ has a strongly unique predual.
\end{thm}

\begin{proof}
By the remarks at the beginning of this section, $\fA\cap\fK$ has a contractive
approximate identity.  So by Lemma~\ref{L:GoodApproxIDs}, we obtain
approximate identities $\{S_n\}$ and $\{I-C_n\}$ for $\fA\cap\fK$ as described.
In particular, $S_n$ converges \sot$*$ to $I$ and $C_n$ converges \sot$*$ to $0$.

Let $E$ be a subspace of $\fA^*$ which norms $\fA$, and so that the closed unit ball
of $\fA$ is $\sigma(\fA,E)$ compact.
Fix $A \in \fA$ with $\|A\|=1$.
Then the sequence $C_n A C_n$ is bounded,
and converges \sot$*$ to $0$.
Since this is a bounded net, it has a subnet $C_\alpha A C_\alpha$
which converges in the $\sigma(\fA,E)$ topology to some element $B$
in the ball of $\fA$.
That is,
\[ \lim \phi \big( C_\alpha A C_\alpha \big) = \phi(B) \qforal \phi \in E .\]

We will show that $B=0$.
Fix $n\ge1$. Then for $\phi\in E$ with $\|\phi\| = 1$,
\begin{align*}
 \big| \phi(  S_n \pm B) \big| &=
 \lim_\alpha \big| \phi(  S_n \pm C_\alpha A C_\alpha ) \big| \\ &\le
 \lim_\alpha  \left\| \begin{bmatrix} S_\alpha & C_\alpha \end{bmatrix}
 \begin{bmatrix}  S_n & 0\\ 0&  \pm A \end{bmatrix}
 \begin{bmatrix} S_\alpha \\ C_\alpha \end{bmatrix} \right\|
 +\|S_n - S_\alpha S_n S_\alpha\| \\ &\le
 1 + \lim_\alpha   \|S_n - S_\alpha S_n S_\alpha\|  =1 .
\end{align*}
Since $E$ norms $\fA$, we conclude that $\| S_n \pm  B \| \le 1$.
Letting $n$ go to infinity, this converges \wot\ to $I \pm B$.
Hence $\| I \pm B \| \le 1$.
Therefore $B=0$.

Fix $\phi\in E$.  We have the decomposition $\phi=\phi_a + \phi_s$.
Note that $A -  C_\alpha A C_\alpha =(1-C_\alpha) A +C_\alpha A(1-C_\alpha)$ is compact
because $1-C_\alpha$ is compact, and this net converges \wot\ to $A$.
Hence
\begin{align*}
 \phi(A) &= \phi(A-B) = \lim \phi \big( A - C_\alpha A C_\alpha \big)\\
 &= \lim \phi_a \big( A - C_\alpha A C_\alpha \big)
 = \phi_a(A) .
\end{align*}
It follows that $\phi=\phi_a$.
This shows that $E$ is contained in $\fA_*$.
Since it separates points, $E=\fA_*$.
\end{proof}

A nest algebra is the set of all operators that are upper triangular
with respect to a fixed chain of invariant subspaces.
All nest algebras have a dense subalgebra of compact operators \cite{Erdos}.
More generally, a CSL algebra is a reflexive algebra of operators containing a masa.
The lattice of invariant subspaces is a sublattice of the projection lattice
of the masa, hence the name commutative subspace lattice (CSL).
The compact operators are weak-$*$ dense precisely when the lattice is
completely distributive \cite{LL}.
When the masa is atomic, one can find an approximate identity for $\fK$
consisting of finite rank projections in the masa.
So in particular, the compact operators in atomic CSL algebras are weak-$*$ dense.
See \cite{Dnest} for more details.  So we obtain:

\begin{cor}
Every completely distributative CSL algebra has a unique predual.
\end{cor}

The algebra $L^\infty(0,1)$ is not completely distributive.  But it still has
a unique predual by Grothendieck's Theorem.  We do not know whether
every CSL algebra has a unique predual.


\end{document}